\documentclass[12pt]{amsart}
\usepackage{amsmath}
\usepackage{amssymb}
\usepackage{amstext}
\usepackage{amscd}
\usepackage[matrix,arrow,ps]{xy}

\newtheorem{teo}{Theorem}[section]
\newtheorem{prop}[teo]{Proposition}
\newtheorem{lem}[teo]{Lemma}
\newtheorem{cor}[teo]{Corollary}
\newtheorem{conj}[teo]{Conjecture}
\newtheorem{exe}[teo]{Example}
\newtheorem{defini}[teo]{Definition}
\newtheorem{ques}[teo]{Question}
\newtheorem{rem}[teo]{Remark}

\newcommand{\PGL}{{\rm PGL}}

\newcommand{\GL}{{\rm GL}}

\newcommand{\Sh}{{\rm Sh}}

\newcommand{\ad}{{\rm ad}}

\newcommand{\CC}{{\mathbb C}}
\newcommand{\RR}{{\mathbb R}}
\newcommand{\ZZ}{{\mathbb Z}}
\newcommand{\QQ}{{\mathbb Q}}
\newcommand{\NN}{{\mathbb N}}

\newcommand{\HH}{{\mathbb H}}
\newcommand{\PP}{{\mathbb P}}

\newcommand{\SSS}{{\mathbb S}}
\newcommand{\AAA}{{\mathbb A}}

\newcommand{\lto}{\longrightarrow}

\def\Fk{\mathfrak{k}}
\def\Fp{\mathfrak{p}}

\def\Fg{\mathfrak{g}}

\def\Fm{\mathfrak{m}}

\newcommand{\cF}{\mathcal{F}}
\newcommand{\cA}{{\mathcal A}}

\newcommand{\cL}{{\mathcal L}}

\newcommand{\cX}{{\mathcal X}}

\newcommand{\ol}{\overline}

\newcommand{\vol}{{\mbox{vol}}}

\newcommand{\wt}{\widetilde}

\title{Hyperbolic Ax-Lindemann theorem in the cocompact case.
\footnote{To appear in Duke Mathematical Journal.}}

\author{Emmanuel Ullmo and Andrei Yafaev}

\date{\today}
\begin{document}

\maketitle

\begin{abstract}
We prove an analogue of the classical Ax-Lindemann theorem in the
context of compact Shimura varieties. 
Our work is motivated by J. Pila's strategy for proving the Andr\'e-Oort conjecture unconditionally.
\end{abstract}

\section{Introduction.}

Pila and Zannier \cite{PZ}  recently gave a new proof of the Manin-Mumford conjecture for abelian varieties.
The ideas involved in this proof have been subsequently used by Pila \cite{Pil} to prove
the Andr\'e-Oort conjecture unconditionally for products of modular curves.
From this work emerged a very nice  and promising strategy for proving the Andr\'e-Oort conjecture
for general Shimura varieties.
One important step in this strategy is a hyperbolic analogue of
a theorem of Ax which is a functional version of a classical result of Lindemann. 
The aim of this paper is to prove
this statement for compact Shimura varieties.

Let us first recall the context of the Andr\'e-Oort conjecture.
 For notations concerning
Shimura varieties and their special subvarieties, we refer to \cite{EY} and references contained therein.

Let $(G,X)$ be a Shimura datum and $X^+$ a connected component of $X$.
We let $K$ be a compact open subgroup of $G(\AAA_f)$
 and $\Gamma := G(\QQ)_+\cap K$
where $G(\QQ)_+$ denotes the stabiliser in $G(\QQ)$ of $X^+$.
Then $S:=\Gamma\backslash X^+$ is a connected component of 
$$
\Sh_K(G,X):=G(\QQ)_{+}\backslash X^{+}\times G(\AAA_{f})/K.
$$

\begin{conj}[Andr\'e-Oort]
Let $Z$ be an irreducible subvariety of $\Sh_K(G,X)$ containing a Zariski-dense set
of special points. Then $Z$ is special.
\end{conj}

This conjecture has recently been proved  under the assumption of the Generalised Riemann Hypothesis for CM fields
(see \cite{UY0} and \cite{KY}). Part of the strategy consists in establishing a geometric 
characterisation of special subvarieties
of Shimura varieties. This criterion says roughly that subvarieties contained in their images by certain Hecke correspondences are special.

The Ax-Lindemann theorem is a functional transcendence result  for the exponential
map $\exp :\CC\rightarrow \CC^{*}$. The following geometric version is due to Ax \cite{Ax}.
It is the special case  of the  theorem of Ax  which correspond to the Lindemann (or Lindemann-Weierstrass)
part of the Schanuel conjecture.

\begin{teo}
Let $n$ be an integer and $V$ be an algebraic subvariety of $(\CC^{*})^{n}$. 
Let $\pi=(\exp,\dots,\exp):\CC^{n}\rightarrow (\CC^{*})^{n}$ be the uniformising map.
A maximal complex algebraic subvariety $W\subset \pi^{-1}(V)$ is a translate
of a rational linear subspace.
\end{teo}

The Hyperbolic Ax-Lindemann conjecture is an analogue, in the context of Shimura varieties,
for the uniformising map $\pi: X^{+}\longrightarrow S$. Via the Harish-Chandra embedding,
 $X^{+}$  has a canonical realisation as a bounded symmetric domain in $\CC^{n}$.
Roughly speaking, an algebraic subvariety of $X^{+}$ is defined as the intersection of an algebraic
subvariety of $\CC^{n}$ with $X^{+}$. We will  give  a precise definition of an irreducible algebraic subset
of $X^{+}$ (see section \ref{reduction}). A maximal algebraic subvariety of an analytic subset $Z$ of $X^{+}$ is then
an irreducible  algebraic subvariety contained in $Z$ and maximal among irreducible
algebraic varieties contained in $Z$.
The main result of this paper is the following theorem.

\begin{teo} \label{main_thm}
We assume that $S$ is compact.
Let $\pi\colon X^+ \lto S$ be the uniformising map and let $V$
be an algebraic subvariety of $S$.
Maximal algebraic subvarieties of $\pi^{-1}V$ are precisely the components of the preimages of weakly special 
subvarieties contained in $V$.
\end{teo}

By slight abuse of language, we will refer to a component of a preimage
of a weakly special subvariety of $S$, as a weakly special subvariety of $X^+$.
The Hyperbolic Ax-Lindemann theorem has the following corollary that we proved without assuming
that $S$ is compact in a recent work \cite{UY}.

Note that most Shimura varieties are compact. For example, there is only one adjoint Shimura
datum defining a non-compact Shimura variety of dimension one, namely $(\PGL_2, \HH^{\pm})$ where 
$\HH^{\pm}$ is the union of upper and lower half planes. This Shimura datum defines modular curves.
All other adjoint Shimura data defining Shimura varieties of dimension one are given by $(F^{*}\backslash B^*, \HH^{\pm})$
where $B^*$ is the algebraic group attached to an indefinite  quaternion algebra $B$ over a totally real field $F$ 
which is split at exactly one real place. There are 
infinitely many of those and they define compact Shimura curves.
A shimura datum $(G,X)$ defines a compact Shimura variety if and only if the adjoint group
$G^{ad}$ of $G$ is $\QQ$-anisotropic.

\begin{teo} \label{char}
An irreducible subvariety $Z$ of $S$ is weakly special if and only if some (equivalently any) analytic component of
$\pi^{-1}Z$ is algebraic in the sense explained above.
\end{teo}

The strategy for proving \ref{main_thm} is as follows.
Let $Y$ be a maximal algebraic subvariety of $\pi^{-1}(Z)$.

First of all, the group $G$ can be assumed to be semisimple of adjoint type and the group $\Gamma$ 
sufficiently small,
hence the
variety $S$ is a product $S = S_1 \times \cdots \times S_r$ where the $S_i$s are associated to
simple factors of $G$.  
Without loss of generality, we assume that $V$ and $Y$ are Hodge generic and furthermore
their images by projections to the $S_i$s are positive dimensional.
It can be seen that  we are reduced  to proving that $V=S$.
These reductions are done in  section \ref{reduction}.

The key step is the following result in hyperbolic geometry.
Let $\cF$ be a fundamental domain for the action of $\Gamma$
on $X^{+}\subset \CC^{n}$. With our assumption ($S$ is compact)
the closure of $\cF$ is a compact subset of $X^{+}$.
 Let $S_{\cF}$ be a finite
system of generators for $\Gamma$ and $l:\Gamma\rightarrow \NN$
be the associated word metric (see section \ref{s.3.1} for details).
Let $C$ be an algebraic curve in $\CC^{n}$  such that  $C\cap \cF\neq \emptyset$.
We prove the following theorem (see  theorem \ref{t1.6}): 
\begin{teo}\label{teo10}
There exists a constant $c>0$ such that for all integers $N$ large enough,
$$
\vert \{\gamma\in \Gamma , C\cap \gamma\cF\neq \emptyset \mbox{ and } l(\gamma)\leq N   \}\vert
\ge e^{cN}.
$$
\end{teo}

 This theorem is then combined with the Pila-Wilkie counting theorem \cite{PW}
 to obtain information about the stabiliser $\Theta_{Y}$ of $Y$ in $G(\RR)$.
In what follows `definable' refers to definable in the o-minimal structure $\RR_{an,exp}$
(see section \ref{ominimal} for the relevent background on o-minimal theory). 
For the main results concerning compact Shimura varieties  we could work with the $o$-minimal structure
$\RR_{an}$. The compactness assumption
on the Shimura variety is only used in section 2, we therefore work with  
the o-minimal structure $\RR_{an,exp}$.
  
We define a certain definable subset $\Sigma(Y)$ of $G(\RR)$ as
$$
\Sigma(Y):=\{g\in G(\RR), \dim(gY\cap \pi^{-1}(Z)\cap \cF)=\dim(Y)\}.
$$
Using the elementary properties of heights,  theorem \ref{teo10},
the maximality of $Y$
and the Pila-Wilkie counting theorem (applied to $\Sigma(Y)$)
we show that $\Theta_{Y}$ contains ``many'' elements of $\Gamma$.
As a consequence we show the following result in section \ref{stab}.

\begin{teo}
The stabiliser $\Theta_{Y}$ contains 
 a positive dimensional $\QQ$--algebraic 
subgroup  $H_{Y}$. Moreover $H_{Y}$ is not a unipotent
group.
\end{teo}

For the last part of the proof we construct some Hecke operators $T_{\alpha}$
with $\alpha\in  H_{Y}(\QQ)$ such that $T_{\alpha}$ has dense orbits in $S$
and such that $\pi(Y)\subset T_{\alpha}(Z)\cap Z$. An induction argument
finishes the proof.

The idea of using Hecke correspondences
already appeared in the proof of the Andr\'e-Oort conjecture under the assumption of the Generalised
Riemann hypothesis mentioned above. The key point is a theorem concerning the monodromy
action on $\pi^{-1}(Z)$ due to Deligne \cite{De0} and Andr\'e \cite{An} in the context of variation
of polarized Hodge structures. Some simple properties of 
 Hecke correspondences then allow us to conclude.

The assumption that $S$ is a compact Shimura variety is only used in the
theorem \ref{teo10} and in some arguments concerning the definability 
of the restriction of the uniformising map $\pi$ to the fundamental 
domain $\cF$. 
 We believe that the conclusion of the theorem \ref{teo10}
could hold even in the non compact case (possibly under certain additional assumptions on the curve $C$). 
With a suitable analogue of \ref{teo10} in the general case,
the hyperbolic Ax Lindemann
conjecture can be proved
for the moduli space of principally polarised abelian  varieties using the main result of
Peterzil  and Starchenko \cite{PS}.

Finally, very recently, Pila and Tsimermann (see \cite{PT1}) announced the proof of the hyperbolic Ax-Lindemann theorem for $\cA_g$
using a strategy somewhat similar to that used in this article. 

\subsection{Acknowledgements.} 

We would like to thank Ngaiming Mok for his important
 help with the proof of theorem \ref{teo10}.
We also had some interesting discussions with Elisabeth Bouscaren, Laurent Clozel and Anand Pillay related to
this paper. It is  a great pleasure to thank the  referees for their careful reading of the paper 
and their questions and comments which  helped us to improve the quality of the presentation of this paper.
We are very grateful to Leonind Parnovsky and Alexander Sobolev for valuable discussions and comments.

\section{Algebraic curves in hermitian symmetric domains. }

The aim of this section is to prove theorem \ref{teo10}.

\subsection{Word metric, Bergman metric.}\label{s.3.1}

Let $X$ be a hermitian symmetric domain (we omit the superscript $+$ in this section for simplicity of
notations) realised
as a bounded symmetric domain in $\CC^{N}$ by
the Harish-Chandra embedding (\cite{Mo}, ch.4).
Let $\overline{X}$ be the closure of $X$ in $\CC^{N}$
for the usual topology. Let $G$ be the the group
of holomorphic isometries of $X$ and $\Gamma\subset G$
be a cocompact lattice in $G$.

Let $x_{0}$ be a fixed base point of $X$ and let $\cF$
be a fundamental domain for the action of $\Gamma$ on $X$
such that $x_{0}\in \cF$.  We may  assume that 
$\cF$ is open, connected  and that its closure $\ol{\cF}$ is a compact subset of $X$ (recall that $\Gamma$ is cocompact).
With our conventions $\Gamma \ol{\cF}=X$ and $\cF\cap \gamma\cF=\emptyset$ for $\gamma\in \Gamma$ such that $\gamma\neq 1$.

Let 
$$
S_{\cF}:=\{\gamma\in \Gamma-\{1\},\ \gamma \ol{\cF} \cap \ol{\cF}\neq \emptyset\}.
$$
Then $S_{\cF}$ is a finite set such that $S_{\cF}=S_{\cF}^{-1}$. Moreover $S_{\cF}$ generates $\Gamma$.

Let $l:\Gamma\rightarrow \NN$ be the word metric on $\Gamma$ relative to $S_{\cF}$.
By definition $l(1)=0$ and  for $\gamma\neq 0$, $l(\gamma)$ is the minimal number of elements of $S_{\cF}$ required to write
 $\gamma$ as a product of elements of $S_{\cF}$. We also define the word distance on $\Gamma$
by 
$$
l(\gamma_{1},\gamma_{2})=l(\gamma_{2}^{-1}\gamma_{1})
$$
 for 
$\gamma_{1}$, $\gamma_{2}$ in $\Gamma$.

We refer to \cite{Mo} chapter 4.1 for definitions and properties
of the Bergman Kernel on a bounded symmetric domain.
Let $K(Z,W)$ be the Bergman kernel on $X$ 
and 
$$
\omega= \sqrt{-1}\partial \overline{\partial} \log K(Z,Z)
$$
be the associated K\"ahler form. Let $g$ be the associated
hermitian K\"ahler metric on $X$ 
and   $d_{X,h}(\ ,\ )$ be  the associated hyperbolic 
distance on $X$. Finally let $d_{e}(\ ,\ )$ be the Euclidean
distance on $\CC^{N}$ and $d_{X,e}$ be the restriction of $d_{e}$
to $X$. 

The following proposition is a classical result saying in Gromov's
terminology that $(\Gamma,l)$ and $(X,d_{X,h})$ are 
quasi-isometric. This is a consequence of proposition 8.19 of \cite{BH}.

\begin{prop}\label{p1.1}
Choose any map
$$
r:X\rightarrow \Gamma
$$
such that $x\in r(x)\overline{\cF}$.

There exist $\lambda\ge 1$ and $C\ge 0$ such that
for all $x,y$ in $X$
$$
\frac{l(r(x),r(y))}{\lambda}-C\leq d_{X,h}(x,y)\leq \lambda \ l(r(x),r(y))+C.
$$
\end{prop}

It should be noted that the assumption that $\Gamma$ is cocompact in this proposition is essential. The conclusion of the  proposition \ref{p1.1} does not hold
when $\Gamma$ is not cocompact: in this case, the distance $d_{X,h}(x,y)$ is  unbounded for $x$ and $y$ varying 
within a fixed  fundamental domain but 
 $l(r(x),r(y)) = 0$.

We will also use the following result on the Bergman kernel on $X$.

\begin{lem}\label{l1.2}
There exists a system of holomorphic coordinates $Z=(z_{1},\dots,z_{n})$ in $\CC^{N}$
and a polynomial $Q(Z,\overline{Z})$ in the variables $$(z_{1},\dots,z_{N},\overline{z}_{1},\dots,\overline{z}_{N})$$
taking real positive values on $X$ and vanishing identically on the boundary $\partial X$ of $X$
such that
$$
K(Z,Z)=\frac{1}{Q(Z,\overline{Z})}
$$
for $Z$ in $X$.
\end{lem}
{\it Proof.} 
The Bergman kernel on a product $X=X_{1}\times X_{2}$ of bounded  symmetric domains
is the product of the Bergman kernels on the $X_{i}$. We may therefore assume that $X$ is irreducible.
The result is then an application of the computation of $K(Z,Z)$. For the classical irreducible domains 
 we refer to (\cite{Mo} ch. 4.3). For the two exceptional irreducible domains we refer
 to \cite{Xu}.
We just  give the following  typical example.

\begin{exe}
Using the notations of Mok (\cite{Mo} ch. 4.2), let
$p$ and $q$ be two positive integers. Let 
$$
X=D^{I}_{p,q}:=\{Z\in M_{p,q}(\CC)\simeq \CC^{pq},\ I_{q}- ^{t}\overline{Z}Z>0\}.
$$
Then 
$$
K(Z,Z)=\det (I_{q}- ^{t}\overline{Z}Z)^{-p-q}
$$
and we take
$$
Q(Z,\overline{Z}):=\det (I_{q}- ^{t}\overline{Z}Z)^{p+q}.
$$
\end{exe}

\begin{lem}\label{l1.4}

There exist positive constants $a_{1}$, $a_{2}$ and $\theta$
depending only of $(X,g)$ and the choice of $x_{0}$ such that
for all $x\in X$
$$
-a_{1}\log d_{X,e}(x,\partial X)-\theta\leq d_{X,h}(x,x_{0})\leq -a_{2}\log d_{X,e}(x,\partial X)+\theta
$$
\end{lem}
{\it Proof.}
 Let $G=KP$ be the Cartan decomposition associated to $x_{0}$.
 Then $K$ is a maximal compact subgroup of $G$ and $K.x_{0}=x_{0}$.
 
 Let $\Delta\subset \CC$ be the Poincar\'e unit disc endowed with the usual Poincar\'e
 metric $g_{\Delta}:=\frac{dz    d\overline{z}}{(1-\vert z\vert^{2})^{2}}$.
 Let $r$ be the rank of the hermitian symmetric domain $X$.
 By the polydisc theorem (\cite{Mo} ch.5, thm. 1) there exists a totally
 geodesic complex submanifold $D$ of $X$ such that $(D,g_{\vert D})$ of $X$ 
is isometric to $(\Delta,g_{\Delta})^{r}$, such that $X=K.D$ and such that
$x_{0}\in D$.

Let $d_{D,h}(\ ,\  )$ be the hyperbolic distance in $D$. Then $d_{D,h}$
is just the restriction to $D$ of $d_{X,h}$.

Let $x\in X$, there exists $k\in K$ such that $k.x\in D$. Then
$$
d_{X,h}(x_{0},x)=d_{X,h}(k.x_{0},k.x)=d_{X,h}(x_{0},k.x)=d_{D,h}(x_{0},k.x).
$$

\begin{lem}\label{hyp-eucl}
There exist positive constants $\alpha$ and $\beta$ such that
 for all $y\in D$,
$$
\log (d_{e}(y,\partial X))\leq \log (d_{e}(y,\partial D)) \leq \alpha   \log(d_{e}(y,\partial X))+\beta
$$ 
\end{lem}

Assuming temporarily the validity of this result ,
we only need to prove lemma \ref{l1.4} in the case where $X=\Delta^{r}\subset \CC^{r}$ is a polydisc.
As the result is independent of $x_{0}$ we may assume that $x_{0}=0$.
In this case the boundary $\partial \Delta^{r}$ of $\Delta^{r}$ is 
$$
\partial \Delta^{r}=\cup_{i=1}^{r} \ol{\Delta_{i,r-1}} \times C_{i}
$$
where $C_{i}$ is the unit circle on the $i$-th factor and $\ol{\Delta_{i,r-1}}$
is the product of the closed Poincar\'e discs of the remaining factors. 
Therefore 
for 
$$
x=(x_{1},\dots, x_{r})=(\rho_{1}e^{i\theta_{1}},\dots,\rho_{r}e^{i\theta_{r}})\in \Delta^{r},
$$
$$
d_{e}(x,\partial \Delta^{r})=\min_{1\leq k\leq r} (1-\rho_{k}).
$$

We recall that 
for $x=\rho e^{i\theta}\in \Delta$ we have
$d_{\Delta,h}(0,x)=\log \frac{1+\rho}{1-\rho}$. Therefore
$$
-\log (d_{e}(x,\partial \Delta))\leq d_{\Delta,h}(0,x)\leq -\log(d_{e}(x,\partial \Delta))+ \log 2
$$

The proof of lemma  \ref{l1.4} is then deduced from the inequalities
$$
\max_{k} (d_{\Delta,h}(x_{k},0))\leq d_{\Delta^{r},h}(x,0)\leq r \max_{k} (d_{\Delta,h}(x_{k},0)).
$$

It remains to prove lemma \ref{hyp-eucl}.

The first inequality follows from the fact that $\partial D\subset \partial X$.

We may assume that $X$ is irreducible.
Then $\partial X$ is a smooth real analytic hypersurface and there is a neighbourhood $U$
of $\partial X$ such that  $d_{e}(z,\partial X)$ is analytic for $z$ in $U$
(see  \cite{KP} theorem 3 and the comments p. 120).

Let $\theta:\Delta^{r}\rightarrow D$ be the isometry given by the polydisc theorem.
The function $d_{e}(\theta(z),\partial X)^{2}$ is continuous on $\Delta^r$, analytic in a
neighbourhood of $\partial \Delta^r$ and its zero set is $\partial \Delta^r$.
Let $C$ be a compact subset of $\Delta^r$ such that outside of $C$, 
$d_{e}(\theta(z),\partial X)^{2}$ is analytic. 
We apply Lojasiewicz Vanishing theorem
(\cite{KP2}, thm 6.3.4 p. 169)
to the function $d_{e}(\theta(z),\partial X)^{2}$ outside  of $C$.
This theorem implies that there exists an integer $q>0$ and a real $c_1>0$ such that
for all $z\in \Delta^r\backslash C$, we have
$$
d_e(\theta(z), \partial X)^2 \geq c_1 d_e(z, \partial \Delta^r)^{q}= c_1 d_e(\theta(z), \partial D)^q
$$
On the compact $C$, the functions $d_e(\theta(z), \partial X)^2$ and
$d_e(z, \partial \Delta^r)^{q}$ are continuous and strictly positive.
Therefore there exists $c_2>0$ such that for $z \in C$,
$$
d_e(\theta(z), \partial X)^2 \geq c_2 d_e(z, \partial \Delta^r)^{q} = c_2 d_e(\theta(z), \partial D)^q
$$
Lemma \ref{hyp-eucl} follows.

\subsubsection{Algebraic curves in bounded symmetric domains.}

We keep the notations of the previous section.
In particular
 $X$ is a hermitian symmetric domain  in $\CC^{N}$
via the Harish-Chandra embedding. As in the last section
we let $\overline{X}$ be the closure of $X$ in $\CC^{N}$.
  Let $C$ be an affine integral
algebraic curve in $\CC^{N}$ such that $C\cap X\neq \emptyset$.

\begin{defini}
We define the following counting functions
$$
N_{C}(n):=\vert \{ \gamma\in \Gamma,  \mbox{ such that } \dim (\gamma.\cF\cap C)=1 \mbox{ and } l(\gamma)\leq n   \}  \vert
$$
and
$$
N'_{C}(n):=\vert \{ \gamma\in \Gamma,  \mbox{ such that } \dim (\gamma.\cF\cap C)=1 \mbox{ and } l(\gamma)= n   \}  \vert.
$$
\end{defini}

The main result we have in view in this section is the following theorem.

\begin{teo}\label{t1.6}
There exists a  positive constant $c$ such that for all $n$ big enough
\begin{equation}
N_{C}(n) 
\ge e^{cn}
\end{equation}
\end{teo}

As $C$ is algebraic, $C$ is not contained in a compact subset of $\CC^{N}$.
Therefore, there exists $b\in C\cap \partial X$ such that for all neighbourhoods $V_{b}$
of $b$ in $\CC^{N}$, $V_{b}\cap C$ is not contained in $\overline{X}$.
As $\partial X$ is a real-analytic hypersurface of $\CC^{N}\simeq \RR^{2N}$, there exists 
a neighborhood $V_{b}$ of $b$ such that
 $$
 C\cap\partial X\cap V_{b}
 $$
is a real analytic curve. 

Let $\alpha$ and $\beta$ two real numbers such that
$0\leq \alpha<\beta\leq 2\pi$. Let $\Delta_{\alpha,\beta}$
be the subset of the unit disc $\Delta$ defined as
$$
\Delta_{\alpha,\beta}:=\{z=re^{i\theta}, 0\leq r<1, \alpha\leq \theta\leq \beta   \}.
$$
Let $\overline{\Delta}$ be the closure of $\Delta$ in $\CC$
and $C_{\alpha,\beta}$ the subset of $\partial \Delta$
$$
C_{\alpha,\beta}:=\{z=e^{i\theta},  \alpha\leq \theta\leq \beta   \}.
$$

We may find  $\alpha$ and $\beta$ with the previous properties and a real analytic
map
$$
\psi: \Delta_{\alpha,\beta}\rightarrow C\cap X
$$
such that $\psi$ extends to a real analytic map from a neighborhood 
of $\Delta_{\alpha,\beta}$ to $C$ such that $\psi(C_{\alpha,\beta})\subset C\cap \partial X$.

\begin{lem}\label{l1.5}
Let $K(Z,W)$ be the Bergman kernel on $X$ and
$$
\omega= \sqrt{-1}\partial \overline{\partial} \log K(Z,Z)
$$
be the associated K\"ahler form.
Let $\omega_{\Delta}=\sqrt{-1}\frac{dz\wedge d\overline{z}}{(1-\vert z\vert^{2})^{2}}$
be the K\"ahler form on $\Delta$ associated to the Poincar\'e metric.

(a) There exists a positive integer
$s$ such that  up to changing $\alpha$ and $\beta$
$$
\psi^{*}\omega=s\omega_{\Delta}+ \eta
$$
for a $(1,1)$-form $\eta$ smooth in a neighbourhood of $C_{\alpha,\beta}$.

(b) Let $d_{e,\Delta}(\ ,\ )$ be the Euclidean distance on $\Delta$.
Changing $\alpha$ and $\beta$ if necessary, there exists
 $\lambda'>0$ and $C'>0$ such that for all $z\in \Delta_{\alpha,\beta}$
$$
\vert \log d_{e}(\psi(z),\partial X)-\lambda' \log d_{\Delta,e}(z,\partial \Delta)\vert\leq C'. 
$$
 
\end{lem}

Let $Q(Z,\overline{Z})$ be the polynomial defined in lemma \ref{l1.2}.
Then
$$
\psi^{*}\omega=-\sqrt{-1}\partial_{z}\partial_{\overline{z}} \log Q(\psi(z),\overline{\psi(z)}).
$$
For $z\in \Delta_{\alpha,\beta}$, $Q(\psi(z),\overline{\psi(z)})$
is a real analytic function  taking positive real values
on $\Delta_{\alpha,\beta}$ and vanishing identically
on 
$$
C_{\alpha,\beta}=\{z=re^{i\theta}\in \overline{\Delta},  1-\overline{z}z=0, \alpha\leq \theta\leq \beta\}.
$$
Therefore 
$$
Q(\psi(z),\overline{\psi(z)})=(1-\overline{z}z)^{s}Q_{1}(z)
$$
for a positive integer $s$ and a real analytic function $Q_{1}$ 
 taking positive real values
on $\Delta_{\alpha,\beta}$ and  non vanishing identically
on $C_{\alpha,\beta}$.  By changing $\alpha$ and $\beta$
we may assume that $Q_{1}$ doesn't vanish in a neighbourhood 
of $C_{\alpha,\beta}$. We then take
$\eta=-\sqrt{-1}\partial_{z}\partial_{\overline{z}} \log Q_{1}(z)$
and use the fact that
$\omega_{\Delta}=-\sqrt{-1}\partial_{z}\partial_{\overline{z}} \log (1-\overline{z}z)$
to conclude the proof of the part (a) of lemma.

The proof of (b) is similar. As $\psi$ is real analytic,
$d_{e}(\psi(z),\partial X)^{2}$ is a real analytic function
positive on $\Delta_{\alpha,\beta}$ and vanishing uniformly
on $C_{\alpha,\beta}$. Therefore
$$
d_{e}(\psi(z),\partial X)^{2}=(1-\overline{z}z)^{2\lambda'}\psi_{1}(z)
$$
for a real analytic function $\psi_{1}$
positive on $\Delta_{\alpha,\beta}$ which doesn't vanish uniformly
on $C_{\alpha,\beta}$. Changing $\alpha$ and $\beta$ if necessary
we may assume that $\psi_{1}$ doesn't vanish on $C_{\alpha,\beta}$.
Therefore $\log(\psi_{1}(z))$ is a bounded function on $\Delta_{\alpha,\beta}$.
This finishes the proof of part (b) of lemma.

\medskip

As a consequence of part (a) of the lemma we obtain the following result.

\begin{cor}\label{c1.6}
Let $\gamma\in \Gamma$ such that $\dim (\gamma.\cF\cap \psi (\Delta_{\alpha,\beta}))=1$.
Then
\begin{equation}
\int_{\gamma.\cF\cap C}\omega=s\int_{\psi^{-1}(\gamma.\cF\cap C)} \omega_{\Delta}+
\int_{\psi^{-1}(\gamma.\cF\cap C)} \eta.
\end{equation}
\end{cor}

\begin{lem}\label{l1.7}
There exists a constant $B$ such that for all $\gamma\in \Gamma$ such that 
$\dim (\gamma.\cF\cap C)=1$,
\begin{equation}
\int_{\gamma.\cF\cap C}\omega \leq B.
\end{equation}
\end{lem}

\begin{proof}
Let $X_{c}$ be the compact dual of $X$. Then 
$X_{c}$ is a projective algebraic variety. Let $\cL$ be the 
dual of the canonical line bundle  endowed with the $G(\CC)$-invariant  metric $\Vert \ \Vert_{FS}$.
Then $\cL$ is ample.
Let $\omega_{FS}$ be the curvature form of
$(\cL,\Vert \ \Vert_{FS})$.

By the Harish-Chandra embedding theorem (\cite{Mo}, ch4.2, thm. 1)
there is a biholomorphism $\lambda$ from $\CC^{N}$ onto a dense open
subset of $X_{c}$. This $\lambda$ is in fact algebraic. To see this let's recall
its definition. We refer to \cite{Wo} p 281 for the following facts.
 There is a commutative subalgebra  $\Fm^{+}$ of $Lie(G_{\CC})$
consisting of nilpotent elements such that $\CC^{N}=\Fm^{+}$. 
Let $M^{+}=\exp (\Fm^{+})$. 
As $\Fm^{+}$ is a nilpotent Lie algebra, the group $M^+$ is algebraic 
(see for example \cite{MiMi}, Cor 4.8, p 276)
The map $\exp:\Fm^{+}\rightarrow M^{+}$
is therefore algebraic. Then $X_{c}=G_{\CC}/P$ for some parabolic subgroup
of $G_{\CC}$ and as $\lambda (m)=\exp(m).P$, the morphism $\lambda$
is algebraic.

We recall that the degree of a subvariety $V$ of 
dimension $d$ of $X_{c}$
with respect to $\cL$ can be computed
as $\deg_{\cL}(V)=\int_{V}\omega_{FS}^{d}$.
For a subvariety $V$ of $\CC^{N}$ we may define
$\deg_{\cL}(V)$ as the degree of the Zariski closure
in $X_{c}$ of $\lambda(V)$.

On the compact set $\ol{\cF}$,
$\omega$ and $\lambda^{*}\omega_{FS}$ are two smooth positive
$(1,1)$-forms. There exists therefore a constant $B_{1}$
such that 
$$
\omega\leq B_{1}\lambda^{*}\omega_{FS}
$$
on $\cF$.

Therefore
$$
\int_{\gamma.\cF\cap C}\omega=\int_{\gamma^{-1} (\gamma \cF\cap C)}\omega
\leq B_{1} \int_{\gamma^{-1} (\gamma \cF\cap C)}\lambda^{*}\omega_{FS}.
$$
Moreover,
$$
\int_{\gamma^{-1} (\gamma \cF\cap C)}\lambda^{*}\omega_{FS}\leq
\int_{\gamma^{-1}\lambda ( C)} \omega_{FS}=\deg_{\cL}(\gamma^{-1}.\lambda(C)))=
\deg_{\cL}(\lambda(C)).
$$

This finishes the proof of the lemma.
\end{proof}

\begin{lem}\label{l1.8}
There exist positive constants $\lambda_{1}$, $\lambda_{2}$ and $D$ such that
for all $z\in \Delta_{\alpha,\beta}$ such that $z\in \psi^{-1}(\gamma.\cF\cap C)$,
$$
\lambda_{1}l(\gamma)-D\leq -\log (1-\overline{z}z)\leq \lambda_{2}l(\gamma)+D
$$
\end{lem}

\begin{proof} 
This
is a combination of proposition  \ref{p1.1}, lemma \ref{l1.4} and the part (b) of the
lemma \ref{l1.5}.
\end{proof}

We now prove theorem \ref{t1.6}.
Let $n$ be an integer and
$$
I_{n}:=\{z\in \Delta_{\alpha,\beta}, e^{-(n+1)}\leq 1-\vert z\vert\leq e^{-n}\}.
$$
Then there exists a $\delta_{1}>0$ such that for all $n$ big enough,
$$
\vol_{\Delta,h}(I_{n})=\int_{I_{n}}\sqrt{-1}\frac{dz\wedge d\overline{z}}{(1-\vert z\vert^{2})^{2}}\ge \delta_{1} e^{n}.
$$

By lemma \ref{l1.8}, there exist positive constants $c_{1}<c_{2}$ such that 
for all $n$ big enough and for all $z\in I_{n}$, $\psi(z)\in \gamma.\cF_{0}$ 
for some $\gamma$ such that
$$
c_{1}n\leq l(\gamma)\leq c_{2}n.
$$

Using corollary \ref{c1.6}
and lemma \ref{l1.7} we see that there exists a constant $B_{1}$
such that  for all $z\in \Delta_{\alpha,\beta}$
with $\psi(z)\in \gamma. \cF$ for some $\gamma\in \Gamma$
$$
\vol_{\Delta,h}(\psi^{-1}(\gamma.\cF\cap C))\leq B_{1}
$$

Combining these two results we see that there exists a positive constant $\delta$
such that for all $n$ big enough
$$
\sum_{c_{1}n\leq k\leq c_{2}n} N'_{C}(k)\ge \delta e^{n}.
$$
This finishes the proof of theorem \ref{t1.6}.

\begin{ques}\label{mainconj}
A natural extension of theorem \ref{t1.6} would be that
for all arithmetic lattices $\Gamma$ of $X$ (a priori not cocompact)
and for any algebraic curve $C$ in $\CC^{n}$ such that 
$C\cap X\neq \emptyset$,
$$
N_{C}(n)\ge e^{cn}
$$
for some positive constant $c$. We have not been able
to prove this  but we believe
that this statement could be true.  
\end{ques}

\section{o-minimality.}\label{ominimal}

\subsection{Preliminaries}

In this section we briefly review the notions of o-minimal structures and their properties that we
will use later on. For details we refer to \cite{VDD}and \cite{S}  as well as  references therein.

\begin{defini}\label{def2.1}
We recall that a subset of $\RR^{n}$ is  semi-algebraic if it is a finite boolean combinations of sets
of the form 
$$
\{(x_{1},\dots,x_{n})\in \RR^{n},  P(x_{1},\dots,x_{n})\ge 0\}
$$
for some polynomial $P\in \RR[x_{1},\dots,x_{n}]$.
A subset of $\CC^{n}$ is   semi-algebraic if it is a semi-algebraic subset of $\RR^{2n}$ identified with $\CC^{n}$ via the real and imaginary parts of the
 coordinates. 
\end{defini}

\begin{defini}\label{def2.2}
A structure over $\RR$ is  a collection $\mathfrak{S}$ of subsets of $\RR^n$ (with $n\in \NN$) which contains all
semi-algebraic subsets and which is closed under Cartesian products, Boolean operations and 
coordinate projections. 
A structure is called o-minimal (standing for `order-minimal') if subsets of $\RR$ belonging to 
$\mathfrak{S}$ are finite  unions of points and intervals. 
Given a structure $\mathfrak{S}$, subsets of $\RR^n$ belonging to $\mathfrak{S}$ are called \emph{definable}.
If $A$ is a subset of $\RR^{n}$ and $B$ a subset of $\RR^{m}$, a function $f\colon A \lto B$
is called definable if its graph in $A\times B$ is definable.
\end{defini}

In this paper we consider the structure $\RR_{an,exp}$.
This structure contains the graph of $exp \colon \RR\lto \RR$ and `an' indicates that this structure contains
the graphs of all restricted analytic functions.
A restricted analytic function is a function $f \colon \RR^n \lto \RR$ such that 
$f(x) = 0$ for $x\notin [-1,1]^n$ and $f$ coincides on $[-1,1]^{n}$ with an analytic function  $\tilde{f}$
defined 
on a neighborhood $U$
of $[-1,1]^n$.
 It was proved by Van Den Dries and Miller \cite{VdM} that the structure $\RR_{an,exp}$ is
o-minimal. Throughout the paper, by definable we mean definable in the structure $\RR_{an,exp}$.

We will  use the Pila-Wilkie counting theorem and its refinement by Pila.
For a set $X\subset \RR^n$ which is definable in some o-minimal structure, we define
the algebraic part $X^{alg}$ of $X$ to be the union of all positive dimensional semi-algebraic subsets of $X$.

Pila ( \cite{Pi}, 3.4) gives the following definition:
\begin{defini}
A semialgebraic block of dimension $w$ in $\RR^{n}$ is a connected definable set $W\subset \RR^{n}$
of dimension $w$, regular at every point, such that there exists a semialgebraic set $A\subset \RR^{n}$
of dimension $w$, regular at every point with $W\subset A$.
\end{defini}

For $x=(x_1,\dots, x_n)\in \QQ^n$, we define the height of $x$ as
$$
H(x) = \max(H(x_1),\dots,H(x_n) )
$$
where for $a,b\in \ZZ\backslash\{0\}$ coprime
$H(\frac{a}{b})=\max(|a|,|b|)$,  and $H(0)=1$.

For a subset $Z$ of $\RR^{n}$ and a positive real number $t$, we
define the set
$$
\Theta(Z,t):=\{ x\in Z\cap \QQ^n, H(x)\leq t\}
$$
and the counting function
$$
N(Z,t) := | \Theta(Z,t)  |.
$$

The following theorem is the main result of Pila-Wilkie \cite{PW}
and its refinement by Pila (\cite{Pil}, thm. 3.6).

\begin{teo}\label{PW}
Let $X\subset \RR^n$ be a definable set in some o-minimal structure.
For any $\epsilon>0$, there exists a constant $C_{\epsilon}>0$ such that
$$
N(X\backslash X^{alg}, t) < C_{\epsilon}t^{\epsilon}
$$ 
and the set $\Theta(X,t)$ is contained in the union of at most 
$C_{\epsilon}t^{\epsilon}$ semialgebraic blocks contained in $X$.
\end{teo}
This theorem was extended by Pila to points with coordinates in arbitrary number fields.
A version for definable families is also available. 
However we will not use these
generalised versions of the theorem.

\section{The setup.}

In this section we set the relevant notations and prove some preliminary lemmas.

Let $(G,X)$ be a Shimura datum and $X^+$ a connected component of $X$.
We let $K$ be a compact open subgroup of $G(\AAA_f)$
 and $\Gamma := G(\QQ)_+\cap K$
where $G(\QQ)_+$ denotes the stabiliser in $G(\QQ)$ of $X^+$.
We let $S:=\Gamma\backslash X^+$. Then $S$ is  a connected component of 
the Shimura variety $\Sh_K(G,X):=G(\QQ)\backslash X\times G(\AAA_{f})/K$.
We will assume that 
$\Gamma$ is a cocompact lattice. In this situation $S$ is a projective variety.
We recall that this is the case if and only if $G$ is $\QQ$-anisotropic
(see \cite{Bo} thm. 8.4). 

Via a faithful rational representation, we view $G_{\QQ}$ as a closed subgroup of some
$\GL_{n,\QQ}$. We assume that $K$ is neat (see 
\cite{Pi} sec. 0.6)
and contained in $\GL_{n}(\widehat{\ZZ})$. With these assumptions
$\Gamma$ is contained in $\GL_{n}(\ZZ)$.

Using Deligne's interpretation of symmetric spaces in terms 
of Hodge theory we obtain variations of polarised $\QQ$-Hodge
structures on $X^{+}$ and $S$. We refer to (\cite{MoMo} sec. 2) for the definitions
of the Hodge locus on $X^{+}$ or $S$. An irreducible analytic subvariety $M$ of $X^{+}$
or $S$ is said to be Hodge generic if $M$ is not contained in the Hodge locus.
If $M$ is not assumed to be irreducible we say that $M$ is Hodge generic if
all the irreducible components of $M$ are Hodge generic.

\subsection{Definability of the uniformization map in the cocompact case.} \label{PS}

In this section we will for simplicity of notations write $X$ instead of $X^+$.

A realisation $\cX$ of $X$ is a real or complex analytic subset of a real or a complex algebraic
variety $\tilde{\cX}$ such that there is a transitive $G(\RR)^+$-action
 on $\cX$ such that for $x_0\in \cX$ the map
$$
\psi_{x_0}: G(\RR)^+\rightarrow \cX
$$
$$
g\mapsto g \cdot x_0
$$
is semi-algebraic and such that $\cX$ is a hermitian symmetric space associated to $G$. 
In particular $\cX$ is endowed with a $G(\RR)^{+}$-invariant complex structure and 
the action of $G(\RR)^{+}$ on $\cX$ is given by holomorphic maps (\cite{He} ch. 8-4 prop. 4.2). 
We will say that the realisation $\cX$ is real or complex if $\wt{\cX}$ is real or complex.

The realisation of Borel  $X_{B}$ of $X$ as a subset of the compact dual $X^{\vee}$ of  $X$
or the realisation $\mathcal{D}$ of $X$ as a bounded domain are complex realisations
in the previous sense. For $X_{B}$ we recall (\cite{Mo} 3.3 thm 1 p. 52) that there is an
algebraic subgroup $P_{\CC}$ of $G_{\CC}$ such that  $X^{\vee}=G(\CC)/P(\CC)$
and $X_{B}=G(\RR)^{+} \cdot P(\CC)$. The action of $G(\RR)^{+}$ on $X_{B}$ is the restriction
to $G(\RR)^{+}$ of the algebraic action of $G(\CC)$ on $X^{\vee}$ by left multiplication
and is therefore semi-algebraic. 

For details on the bounded realisation $\mathcal{D}$ of $X$ we refer to \cite{BB} 1.4.
Let $\Fg$ be the Lie algebra of $G(\RR)$ and $\Fg=\Fk\oplus \Fp$ be a Cartan decomposition.
Then we have the decompositions  $\Fg_{\CC}=\Fk_{\CC}\oplus \Fp_{\CC}$ and $\Fp_{\CC}=\Fp^{+}\oplus\Fp^{-}$
where  $\Fp^{+}$ and $\Fp^{-}$ are   commutative nilpotent
Lie algebras. Then  $\mathcal{D}$ is realised as a subset of $\Fp^{+}$.
Let $P^{+}$, $P^{-}$ and $K_{\CC}$ be the algebraic subgroups of $G(\CC)$ corresponding
to $\Fp^{+}$, $\Fp^{-}$ and $\Fk_{\CC}$ respectively. 
Then there is a Zariski open subset $\Omega$ of
$G(\CC)$ containing $G(\RR)^{+}$ such that $\Omega=P^{+}K_{\CC}P^{-}$.  
We usualy write $g=g^{+}kg^{-}$ the associated decomposition of 
an element of $\Omega$.
The map 
$$
\zeta: \Omega\rightarrow \Fp^{+}
$$
$g=g^{+}kg^{-}\mapsto \log(g^{+})$ is algebraic as $\Fp^{+}$
is nilpotent.
The action of $G(\RR)^{+}$ on $\mathcal{D}\subset \Fp^{+}$ is given
by 
$$
g.p=\zeta (g\cdot\exp(p))
$$
 and is therefore semi-algebraic as again $\Fp^{+}$ is nilpotent.
 This shows that  $\mathcal{D}$ is a realisation of $\cX$ in the previous sense. 

For the classical bounded domains  an explicit description is given in \cite{Mo} ch.4. 
As an example let us mention (using the the notations of loc. cit.)
$$
D^{I}_{p,q}:=\{ Z\in M_{p,q}(\CC)\simeq \CC^{pq}: I_{q}-\  ^{t}\overline{Z}Z>0\}
$$
which can be seen to be semialgebraic by a direct use of the definition
\ref{def2.1}. The associated group is $\rm{Su}(p,q)$ and the semi-algebraic action of a matrix
$\begin{pmatrix}
A&B\cr C&D
\end{pmatrix}$ of ${\rm Su}(p,q)$ on $D^{I}_{p,q}$ is given by
 $$
\begin{pmatrix}
A&B\cr C&D
\end{pmatrix}.Z=(AZ+B)(CZ+D)^{-1}
$$

The other cases are similar. For explicit descriptions of the two exceptional bounded symmetric
domains we refer to \cite{Xu}.

The theory of Cayley transforms
(\cite{BB} 1.6) shows that the unbounded realisations of $X$ in $\Fp^{+}$ including the classical 
Siegel domains of type I, II and III are complex realisations as defined above.

We recall that we fixed  a faithful representation $G\hookrightarrow \GL_n$ over $\QQ$.
Let $x_0\in X$ and $K_{x_{0}}$ be the stabiliser of $x_0$ in $G(\RR)$.
This induces an identification of $X$ with $G(\RR)/K_{x_{0}}$. Let $K_{\infty}$
be a maximal compact subgroup of $\GL_n(\RR)$ containing $K_{x_{0}}$.
Notice that $\GL_n(\RR)/K_{\infty}$ can be identified with the set $P_{n}$ of positive definite symmetric matrices and is
therefore a semi-algebraic set. Let $S_{n}$ be the algebraic set of symmetric matrices.

The inclusion $X \hookrightarrow \GL_n(\RR)/K_{\infty}$
is a real realisation of  $X$ as a subset $X_{ps}$ ({\it ps} stands for 'positive symmetric') of $\GL_n(\RR)/K_{\infty}=P_{n}\subset S_{n}$ as the map
$$
\phi_{x_{0}}: \GL_{n}(\RR)\rightarrow P_{n}
$$

$$g\mapsto g\cdot x_{0}=g^{t}g$$
is semi-algebraic
and  $X_{ps}=\phi_{x_{0}}(G(\RR))$.

Let $\cX_1$ and $\cX_2$ be two realisations of $X$.
Then an isomorphism of realisations is an analytic
$G(\RR)^+$-equivariant  analytic  diffeomorphism
$\psi: \cX_1\rightarrow \cX_2$. Two realisations of $X$
are always isomorphic as they are isomorphic to $X$.

\begin{prop}\label{prop2.2}
Let $\cX$ be a realisation of $X$. Then $\cX$ is semi-algebraic.
The map 
$$
\phi: G\times X\longrightarrow X
$$
given by $ \phi(g,x)=g.x$ is semi-algebraic.
Let $\psi:\cX_1\rightarrow \cX_2$ be an isomorphism of 
realisations of $X$. Then $\psi$ is a semi-algebraic map.
\end{prop}
\begin{proof}
Fix $x_0\in \cX$, then $\cX$ is defined by
the formula
$$
\cX:=\{g.x_0, \ g\in G(\RR)^+\}
$$
for a semi-algebraic action of $G(\RR) ^+$ so is semi-algebraic.

Let $Gr(\phi)\subset G\times \cX\times \cX$ be the graph of $\phi$.
Let $$\phi_{1}:G\times G\rightarrow G\times G\times G$$
be the map $\phi_{1}(g,\alpha)=(g,\alpha, g\alpha)$. Then 
$\phi_{1}$ is semi-algebraic as $G$ is an algebraic group.

Let $$\phi_{2}: G\times G\times G\rightarrow G\times \cX\times \cX$$
be the map $\phi_{2}(h_{1},h_{2},h_{3})=(h_{1},h_{2}.x_{0},h_{3}.x_{0})$.
The maps  $\phi_{2}$ is semi-algebraic by our definition
of a realisation. Therefore 
$$
Gr(\phi)=\phi_{2}\phi_{1}(G\times G)
$$ 
is semi-algebraic hence $\phi$ is semi-algebraic.

As $\psi$ is $G(\RR)^+ $-equivariant the graph $Gr(\psi)$ of $\psi$
is
$$
Gr(\psi):=\{(g.x_0,g. \psi(x_0))\in \cX_1\times \cX_2, \ g\in G(\RR)\}
$$ 
which is semi-algebraic as the actions of $G(\RR)^+$ on $\cX_1$
and $\cX_2$ are semi-algebraic.
\end{proof}

Let  $X^+$ and $\Gamma$ be as before, recall that $\Gamma$ is cocompact.
There are many choices for the fundamental domain in $X^+$ for the action of $\Gamma$. 
We recall that a fundamental set $\Omega$ in $X^{+}$ for the action of $\Gamma$ is by definition 
a subset of $X^{+}$ such that $\Gamma \Omega=X^{+}$ and such that 
$$
\{\gamma\in \Gamma, \gamma\Omega\cap\Omega\neq \emptyset \}
$$
is a finite set. A fundamental set for the action of $\Gamma$ on $G(\RR)$ is defined in the same way.
A fundamental domain in $X^{+}$ is as before an  open set  $\cF$ such that $\Gamma \overline{\cF}=X^{+}$ and such
that for all $\gamma\in \Gamma$ with $\gamma \neq 1$ we have $\cF\cap \gamma \cF=\emptyset$. 

For our purposes, we make the following choice.

\begin{prop} \label{fundomain}
Let $\cX$ be any realisation of $X^{+}$.
There exists a semi-algebraic fundamental set $\Omega$ in $\cX$ for the action of $\Gamma$ such that
$\overline{\Omega}$ is compact. There is a connected fundamental domain $\cF$ contained in  $\Omega$ which is
definable in $\RR_{an}$. 
\end{prop}
\begin{proof}

As two realisations of $X^{+}$ are always isomorphic the proposition is independent of the realisations
of $\cX$ by the proposition \ref{prop2.2}. We will use the realisation $X_{ps}$ of $X^{+}$.

By Theorem 4.8 of \cite{PR}, we can choose
$b_1,\dots , b_r \in \GL_n(\ZZ)$ and a semi-algebraic subset $\Sigma$ of $\GL_n(\RR)$ (as one can check using
it's definition p. 177 of \cite{PR}),
such that $\Sigma K_{\infty}=\Sigma$ and such
 that
$$
 \Omega:=\phi_{x_{0}} ((\cup_{i=1}^r b_i \Sigma )\cap G(\RR))
$$
is a fundamental set for $\Gamma$ acting on $X_{ps}$. 
Note that the choice of $a$ in the theorem 4.8 of \cite{PR} correspond to our
choice of $x_{0}$.
As $\Gamma$ is cocompact, $\overline{\Omega}$ is compact.

Let us now define
$$
\cF = \{x \in X_{ps}, \   \   d_{X_{ps},h}(x,x_0)<d_{X_{ps},h}(\gamma x, x_0), \forall \gamma\in \Gamma; \gamma\neq 1\}.
$$
One can show that $\cF$ is connected in the following way.
For all $\gamma\in \Gamma$ with $\gamma\neq 1$ we define 
$\cF_{\gamma}= \{x \in X_{ps}, \   \   d_{X_{ps},h}(x,x_0)<d_{X_{ps},h}(\gamma x, x_0)\}$.
Then as $x_{0}\in \cF_{\gamma}$ and $\cF=\cap_{\gamma\in \Gamma-\{1\}}\cF_{\gamma}$
 we just need to show that $\cF_{\gamma}$ is connected. A simple geometric argument
 shows that if $x\in \cF_{\gamma}$ then the geodesic arc joining $x_{0}$ to $x$
 is contained in $\cF_{\gamma}$. 

We may enlarge $\Sigma$ and choose the 
$b_{i}$ such that  $\cF\subset \Omega$. We assume therefore that 
$\cF\subset \Omega$.
The set $\cF$ is a fundamental domain for the action of $\Gamma$ on $X_{ps}$.
We claim that $\cF$ is  definable in $\RR_{an}$.
As $\Gamma$ is cocompact and $\Omega$ is a semi-algebraic set   containing $\cF$
we only need to check the inequalities for a finite number of elements $\gamma$.
The function $x \lto d_{X_{ps},h}(x,x_0)$ is a real analytic function,
 it follows that the set $\cF$ is definable in $\RR_{an}$.
\end{proof}

We'll need the following statement which is probably well-known to the experts:

\begin{prop}\label{p.4.1} Let $\cX$ be any realisation of the hermitian symmetric space $G(\RR)/K_{x_{0}}$.
Let 
$$
\pi_{1}:  \cX \longrightarrow S:=\Gamma\backslash \cX
$$
be  the uniformising map. Then there exists a definable 
fundamental relatively compact domain $\cF_{1}$ of $\cX$ (in $\RR_{an}$) such that 
  the restriction of $\pi_{1}$ 
to $\overline{\cF_{1}}$ is definable in $\RR_{an}$ (hence in particular in $\RR_{an,exp}$ ).
\end{prop}
\begin{proof}

Using proposition \ref{prop2.2} we see that the result is independent of the choice
of the realisation $\cX$. 

We prove the result for the bounded realisation $\mathcal{D}$. 
We let $\pi:\mathcal{D}\rightarrow S=\Gamma\backslash \mathcal{D}$ be the uniformising map. 
In our situation $S$ is a smooth  projective variety. By the main result of \cite{BB},
we have an algebraic embedding 
$$
\Psi=(\psi_{0},\psi_{1},\dots,\psi_{n}):S\longrightarrow \PP^{n}_{\CC}.
$$
As such $\Psi$ is definable in $\RR_{an}$.
 
 For all $0\leq i\leq n$ the map 
 $\psi_{i} \pi$ is holomorphic therefore its 
 restriction to the compact semi-algebraic set $\Omega$ (as in \ref{fundomain}) is definable in $\RR_{an}$.
As $\cF$ is a definable subset of $\Omega$,  $\psi_{i} \pi$ is definable when restricted to $\cF$.
As a consequence the restriction of $\Psi \pi$ to $\cF$ is definable in $\RR_{an}$
and we may conclude that the restriction of $\pi$ to $\cF$ is definable in $\RR_{an}$.

\end{proof}

\begin{rem}{\rm

The  generalisation of proposition \ref{p.4.1}
(definability in $\RR_{an,exp}$ instead of $\RR_{an}$)
for the moduli space $\cA_{g}$ of principally polarised
abelian varieties of dimension $g$ is given by Peterzil and Starchenko \cite{PS}.
We believe that the natural extension of their  results
to Shimura varieties of abelian type could be deduced
from the result for $\cA_{g}$. The exceptional Shimura varieties
would require new ideas.

}
\end{rem}

\subsection{Algebraic subsets of bounded symmetric domains.}\label{reduction}

We refer to section 2 of \cite{UY} for relevant definitions and properties of weakly special subvarieties. 
Let $\cX \subset \wt{\cX}$ be a complex realisation of $X^+$.

Let $\widehat{Y}$ be an algebraic subvariety of $\wt{\cX}$.  Proposition \ref{prop2.2} implies that
$\widehat{Y}\cap \cX$ is semi-algebraic (as a  variety over the reals) and complex analytic.
Using \cite{FL} (section 2), we see that $\widehat{Y}\cap \cX$ has only finitely many irreducible 
analytic components  and that these components are semi-algebraic.

 An irreducible algebraic subvariety of $\cX$ is then defined as an irreducible analytic  component of  
the intersection of $\cX$ with a closed algebraic subvariety of $\wt{\cX}$. An algebraic subvariety of $\cX$
 is defined to be a finite union of irreducible algebraic subvarieties 
of $\cX$.

Let $Z$ be a closed analytic subset of $\cX$.
We define the complex algebraic locus $Z^{ca}$ of $Z$  as the union of  the positive dimensional  algebraic subvarieties of $\cX$
which are contained in $Z$. We define the semi-algebraic locus $Z^{sa}$ of $Z$ as the union of the connected 
positive dimensional semi-algebraic sets which are contained in $Z$. Then by the previous discussion
$Z^{ca}\subset Z^{sa}$. Lemma 4.1 of \cite{PT} implies in fact that
\begin{equation}\label{eq300}
Z^{ca}= Z^{sa}.
\end{equation}

Let $V$ be an irreducible algebraic subvariety of $S$, let $\wt{V}$ be the preimage 
of $V$ in $\cX$. An irreducible algebraic subvariety  $Y$ of $\wt{V}$
is said to be maximal if $Y$ is maximal among irreducible algebraic subvarieties of $\wt{V}$.

Let $Y$ be a maximal irreducible algebraic subvarieties of $\wt{V}$.
We let $\cF$ be a fundamental set for the action of $\Gamma$ on $\cX$ as in section \ref{PS}
such that $\dim(Y\cap \cF) = \dim(Y)$. 
In our situation $\cF$ is an open subset of $\cX$ and by our assumption on $\Gamma$, the closure $\ol{\cF}$
is a compact subset of $\cX$.

We can now state the general hyperbolic Ax-Lindemann conjecture.
Using that isomorphisms of realisations are given by semi-algebraic maps (see \ref{prop2.2})
and equality \ref{eq300},
 we see  that this statement is independent of the complex realisation of $X^+$.

\begin{conj}
Let $S$ be a Shimura variety and $\cX$ a complex realisation of $X^+$. Let $\pi \colon \cX \lto S$
be the uniformisation map and let $V$ be an algebraic subvariety of $S$.
Maximal algebraic subvarieties of $\pi^{-1}V$ are weakly special subvarieties.
\end{conj}

\begin{rem}{\rm

The Hyperbolic Ax-Lindemann conjecture
for a not necessary compact Shimura variety 
would essentially be a consequence of a positive answer
to the question \ref{mainconj} and a generalisation
of the definanability of the uniformisation map to non compact situations.

As noted before,  a proof of the Ax-Lindemann conjecture for $\cA_{g}$
has been announced by Pila and Tsimerman \cite{PT1}.
They provide a positive answer to a variant 
of the question \ref{mainconj}.

Some applications of the hyperbolic Ax-Lindemann conjecture
to the set of positive dimentional special subvarieties 
of subvarieties of Shimura varieties and an inconditional 
proof of the Andr\'e-Oort conjecture for subvarieties
of projective Shimura varieties in $\cA_{6}^{n}$ are given
in the recent preprint by the first named author in \cite{Ullmo}.

}\end{rem}

For the purposes of proving Theorem \ref{main_thm}, we can make the following simplifying assumptions.
\begin{lem} \label{reduction1}
Without loss of generality we can assume 
that $G$ is semisimple of adjoint type and that $Y$ (and hence $V$) is Hodge generic.
\end{lem}
\begin{proof}
Let us first check that we can assume that $V$ and $Y$ are Hodge generic.
To assume that $V$ is Hodge generic, it suffices to replace $S$ by the smallest Shimura variety containing $V$.
This amounts to replacing $X^+$ by a certain symmetric subspace and hence does not alter the property of $Y$  being algebraic.
Suppose $\pi(Y)$ is contained in a smaller special subvariety $S'\subset S$. Then one can replace $S$ by $S'$ and $V$ by an irreducible component of $V\cap S'$
and hence assume that $Y$ is Hodge generic.
Notice that the conclusion of \ref{main_thm} is independent of the choice of the compact open subgroup $K$.
Indeed, replacing $K$ by a smaller subgroup does not change $Y$.

Let $G^{\ad}$ be the quotient of $G$ by its centre.
Let $\pi^{\ad}\colon G\lto G^{\ad}$ be the natural morphism.
Recall that $X^+$ is a connected component of the $G(\RR)$-conjugacy class of a 
morphism $x\colon \SSS \lto G_\RR$. Let $X^{\ad}$ be the $G^{\ad}(\RR)$-conjugacy class of
$\pi^{\ad}\circ x$. Then $(G^{\ad},X^{\ad})$ is a Shimura datum and 
$\pi^{\ad}$ induces an isomorphism between $X^+$ and $X^{+\ad}$, a connected component
of $X^{\ad}$. A choice of a compact open subgroup $K^{\ad }$ containing $\pi^{\ad}(K)$ induces
a finite morphism of Shimura varieties $\Sh_K(G,X)\lto \Sh_{K^{\ad}}(G^{\ad},X^{\ad}) $ and 
a subvariety is special if and only if its image is special.
It follows that we can replace $V, \wt{V},Y$ by their images and, respectively
$G,X,K$ by $G^{\ad},X^{\ad}, K^{\ad}$ and hence assume that $G$ is semisimple of adjoint type.
\end{proof}

Since the group $G$ is now assumed to be adjoint, we can write
$$
G = G_1 \times \cdots \times G_r
$$
where the $G_i$s are $\QQ$-simple factors of $G$.
The group $K$ is assumed to be neat and to be a product $K = \prod_p K_p$ where $K_p$s are compact open subgroups of $G(\QQ_p)$. Furthermore, each $K_p$ is assumed to be the product
$$
K_p = K_{p,1}\times \cdots \times K_{p,r}
$$
with $K_{p,i}$s compact open subgroups of $G_i(\QQ_p)$.
The group $K$ is a direct product $K = K_1 \times \cdots \times K_r$
where $K_i = \prod_p K_{i,p}$.
The group $\Gamma$ is a direct product
$$
\Gamma = \Gamma_1 \times \cdots \times \Gamma_r
$$
where $\Gamma_i = K_i \cap G_i(\QQ)^+$.

This induces a decomposition of $S$ as a direct product
$$
S = S_1 \times \cdots \times S_r
$$
where $S_i = \Gamma_i\backslash X_i^+$.

From now on we assume that $V$ and $Y$ are Hodge generic.
We now make the next simplifying assumption.

\begin{lem} \label{reduction2}
Without loss of generality, we may assume that $V$ is not of the form  $S_{1}\times \cdots \times S_t \times V_1$
where $t\geq 1$ and $V_1$ is a subvariety of  $S_{t+1}\times \cdots \times S_r$.

Without loss of generality, we may assume that $Y$ is  not of the form $\{x\}\times Y_1$
where $x$ is a point of a product $X_1^+\times \cdots \times X_t^+$ for $t\geq 1$.
\end{lem}

\begin{proof}
Suppose that $V$ is of this form. Then $\wt{V}$ is of the form $X_1^+\times \cdots \times X_{t}^+ \times \wt{V}_1$. As $Y$ is a maximal algebraic subset of $\wt{V}$, $Y$ is of the form
 $Y=X^+_1\times \cdots \times X^+_t \times Y_1$ where $Y_1$ is a maximal algebraic subset of $V_1$.
Hence we are reduced to proving that $Y_1$ is weakly special.
Notice also that for similar reasons, we can assume that $V$ is not of the form $\{ x \} \times V_1$ where
$x$ is a point of a product of a certain number of $S_i$s.

As for the second claim,  if $Y$ is of this form, then proving that $Y$ is weakly special is equivalent to proving that 
$Y_1$ is weakly special.
\end{proof}

\begin{rem}
Note that with these assumptions we need in fact  to prove that $Y=X^{+}$.
 As $Y$ should be
both prespecial and Hodge generic, a simple argument shows that 
$Y$ should be of the form 
$$
Y=X_{1}^{+}\times \dots\times X_{t}^{+}\times \{x_{t+1}\}\times\dots\times \{x_{r}\} 
$$
for some Hodge generic points $x_{i}$ of $X_{i}^{+}$ for $t+1\leq i\leq r$.
By the assumption of the previous lemma we see that we should have $Y=X^{+}$.

We'll prove in fact that $\wt{V}=X^{+}$,
but this implies that $\wt{V}$ is algebraic and as $Y$ is maximal algebraic in 
$\wt{V}$ we can conclude that $Y=\wt{V}=X^{+}$.
\end{rem}

We now recall some properties  of the  monodromy attached to a Hodge generic subvariety of $S$.
Let $Z$ be an irreducible Hodge generic algebraic subvariety of $S$
all of whose projections to $S_i$s are positive dimensional.

Fix a Hodge generic point $s$ of $Z$ and let $x$ be a point of $\pi^{-1}(s)$.

Fix a faithful rational representation $\rho\colon G\hookrightarrow \GL_n$.
The choice of a $\Gamma$-invariant lattice induces
a variation of polarisable Hodge structures on $S$.
We then obtain the monodromy representation 
$$
\rho^{mon}_s \colon \pi_1(Z^{sm},s)\lto \GL_n(\ZZ)
$$
where $Z^{sm}$ is the smooth locus of $Z$.
Let $\Gamma'$ be the image of $\rho^{mon}_s$.
Let $H^{mon}$ be the neutral connected component of the Zariski closure of
$\rho^{mon}_s$ in $\GL_{n\QQ}$.
By a theorem of Deligne-Andr\'e (see for example \cite{MoMo}, theorem 1.4),
the group $H^{mon}$ is a normal subgroup of $G$.
In our case ($G$ of adjoint type), $H^{mon}$ is a product of simple factors of $G$ and,
after reodering, we can write
$$
G = H^{mon} \times G_1
$$
Proposition 3.7 of \cite{MoMo} then shows that $V = S^1 \times V'$ where
$S^1$ is the product of the $S_i$s corresponding to $H^{mon}$ and $V'$ is a subvariety of the remaining factors. 
According to \ref{reduction2}, we have
$$
H^{mon} = G
$$

Via $\rho$, we view $\Gamma'$ is a subgroup of $\Gamma$.
Let $C(Z)$ be the maximum of the constants $n$ from theorems 5.1 and 6.1 of \cite{EY}.
We call such a constant $C(Z)$ the \emph{Nori constant} of $Z$. 
The next two propositions recall the main properties of the constant $C(Z)$.

\begin{prop}\label{irred}
Let $g\in G(\QQ)^+$ and $p>C(Z)$ a prime such that for any prime $l\not= p$, the image $g_l$ 
of $g$ in $G(\QQ_l)$ is contained in $K_l$, then $T_gZ$ is irreducible.
\end{prop}
\begin{proof}
This is a direct application of the  theorem 5.1 of \cite{EY}.
\end{proof}

\begin{prop}\label{characterisation}
Let $g$ be an element of $G(\QQ)^+$ such that the following holds.
The images of $g_p$  in $G_i(\QQ_p)$ for $i=1,\dots , t$  are not contained in compact open subgroups
while the images of $g_p$ in $G_{t+1}(\QQ_p),\dots, G_r(\QQ_p)$ are contained $K_{p,i}$

Suppose that $ T_gZ=Z$ and $T_{g^{-1}}Z=Z$.
Then $Z= S_1 \times \cdots \times S_t \times Z'$ where 
$Z'$ is a subvariety of $S_{t+1}\times \cdots \times S_r$.
\end{prop}
\begin{proof}

For every $x\in Z$, the $T_{g}+T_{g^{-1}}$ orbit
of $x$ is contained in $Z$.
Let $(x_1,\dots, x_r)$ be a point of 
$Z$ (where for all $i$, $x_i$ is a point of $S_i$). 
By theorem 6.1 of \cite{EY}, the closure of the $T_{g}+T_{g^{-1}}$ orbit of $x$
is 
$$
S_1 \times \cdots \times S_t \times \{x_{t+1}\}\times \dots\times  \{x_r\}.
$$
Hence $Z = S_1\times \cdots \times S_t \times  Z'$ where $Z'$  is the image of 
$Z$ in 
$$
S_{t+1}\times \cdots \times S_r.
$$
\end{proof}

\section{Stabilisers of maximal algebraic subsets.}\label{stab}

We keep the notations of the previous section.
View $G(\RR)$ as an algebraic (hence definable) subset of some $\RR^m$
and $X$ as a semi-algebraic (in particular definable) subset of some $\RR^k$.

\begin{lem} \label{step1}
Let $Y$ be a maximal irreducible algebraic subset of $\wt{V}$.
Define
$$
\Sigma(Y) = \{ g\in G(\RR) : \dim(gY \cap \wt{V} \cap \cF) = \dim(Y) \}  
$$

(a) The set $\Sigma(Y)$ is definable and for all $g\in \Sigma(Y)$,  $g Y \subset \wt{V}$.

(b) For all $\gamma\in \Sigma(Y)\cap \Gamma$, $\gamma.Y$ is a maximal algebraic subvariety of $\wt{V}$.
\end{lem}
\begin{proof}

The set $\cF\cap \wt{V}$ is definable by proposition \ref{p.4.1}.
For each $g\in G(\RR)$, the set $gY \cap \wt{V} \cap \cF$ is a definable subset 
of $\RR^n$. The definability of $\Sigma(Y)$ is now a
consequence of the second part of proposition \ref{prop2.2} and  proposition 1.5 of \cite{VDD}.

Let $g\in \Sigma(Y)$, by definition of $\Sigma(Y)$
$$
g Y \cap \cF \subset \wt{V}.
$$
Both $g Y$ and $\wt{V}$ are analytic varieties, therefore the
above inclusion implies that $gY \subset \wt{V}$. This finishes the proof
of the first part of lemma.

Let $\gamma\in \Sigma(Y)\cap \Gamma$. Then $\gamma.Y\subset \wt{V}$
by the previous result. Let $Y'\subset \wt{V}$ be an algebraic set containing $\gamma.Y$.
Then $\gamma^{-1}.Y'\subset \tilde{V}$. As $Y\subset \gamma^{-1}.Y'$,
$\gamma^{-1}Y'\subset \wt{V}$ and by maximality of $Y$, $Y'=\gamma.Y$.

\end{proof}

\begin{lem} \label{intersection}

Let $Y$ be a maximal algebraic subset of $\wt{V}$.
Define

$$
 \Sigma' (Y)= \{ g \in G(\RR) : g^{-1}\cF \cap Y \not= \emptyset \}
$$
Then
$$
\Sigma(Y)\cap \Gamma = \Sigma'(Y)\cap \Gamma
$$
\end{lem}
\begin{proof}
The inclusion $\Sigma(Y)\subset \Sigma'(Y)$ is obvious.

Let $\gamma \in \Sigma'(Y)\cap \Gamma$. Then $\gamma^{-1}\cF \cap Y \not= \emptyset$.
The set $\gamma^{-1} \cF$ is an open subset of $\RR^n$, hence 
$\dim( \gamma^{-1} \cF \cap Y) = \dim(Y)$.
Translating by $\gamma$, we get $\dim(\gamma Y\cap \cF) = \dim(Y)$.
Therefore
$$
\dim(\gamma Y \cap \wt{V} \cap \cF) = \dim(Y)
$$
i.e. $\gamma \in \Sigma (Y)\cap \Gamma$.
Note that that we have used in an essenatial way that $\wt{V}$ is $\Gamma$-invariant and hence
$\gamma Y \subset \wt{V}$ for any $\gamma \in \Gamma$. In particular 
$\gamma Y \cap \cF = \gamma Y \cap \wt{V} \cap \cF$ for any $\gamma \in \Gamma$.
\end{proof}

We recall that  
 we defined in section \ref{s.3.1} a finite set 
$S_{\cF}$ generating $\Gamma$ and the associated
word metric $l:\Gamma\rightarrow \NN$. 

\begin{prop}\label{p.5.3}
Let $N_{Y,\Gamma}(N):=\vert \{\gamma\in \Gamma\cap \Sigma(Y), l(\gamma)\leq N\}$. There exist
a constant $c>0$ such that for all $N$ big enough 
\begin{equation}
N_{Y,\Gamma}(N)\ge e^{cN}.
\end{equation}
\end{prop}
\begin{proof}
As $Y$ is semi-algebraic and maximal in $\wt{V}$, $Y$ contains a component of the intersection of an algebraic curve
$C$ with $X$. For $\gamma\in \Gamma$ the condition $\gamma\cF\cap C\neq \emptyset$
implies that $\gamma\cF\cap Y\neq \emptyset$. The result is then a consequence of the previous lemma and
the part b of theorem \ref{t1.6}.
\end{proof}

Let $\Theta_{Y}$ be the stabilizer of $Y$ in $G(\RR)$. 
The main result of this part is
the following statement.

\begin{teo}\label{t.5.5}
Let $H_Y$ be the neutral component  
 of  the Zariski closure of $\Gamma\cap \Theta_{Y}$.
Then $H_Y$ is a non-trivial reductive $\QQ$-algebraic group.
\end{teo}

The proof will use the following two lemmas.
\begin{lem}\label{l.5.6}
Let $W$ be a semi-algebraic block of $\Sigma(Y)$ containing some $\gamma\in \Sigma(Y)\cap \Gamma$. Then 
$$
W\subset \gamma \Theta_{Y}.
$$
\end{lem}
\begin{proof}

By the part (b) of lemma \ref{step1}, we know that
 $\gamma.Y$ is a maximal algebraic subvariety of $\wt{V}$. 
Let $U_{\gamma}$ be an open connected semi-algebraic subset of $W$ containing $\gamma$.
Then $U_{\gamma} \cdot Y$ is semi-algebraic containing $\gamma\cdot Y$.
By maximality of $\gamma \cdot Y$ and the equality  (\ref{eq300}) applied to $\tilde{V}$, we have
$$
U_{\gamma} \cdot Y = \gamma \cdot Y
$$
Let $w\in W$. We claim that there exists a connected semi-algebraic open subset $U$ of $W$
which contains $w$ and $\gamma$.
Indeed, $W$ is path connected, we choose a compact path between $w$ and $\gamma$.
We can cover this path with a finite number of open semi-algebraic subsets of $W$. Their union is the desired subset $U$.
Then, by the previous argument,
$$
U\cdot Y = \gamma \cdot Y = w\cdot Y
$$
Therefore $\gamma.Y=W.Y$ and 
$W\subset \gamma\Theta_{Y}$. 
\end{proof}

We define the height $H(\alpha)$  of an element $\alpha$ of $\GL_n(\ZZ)$ as the maximum
of the absolute values of the coefficients of $\alpha$. The triangle inequality shows that
if $\alpha, \beta \in \GL_n(\ZZ)$, then
\begin{equation}\label{eq.5}
H(\alpha\beta) \leq n H(\alpha) H(\beta)
\end{equation}
and 
\begin{equation}\label{eq.6}
H(\alpha^{-1}) \leq c_{n} H(\alpha)^{n-1}
\end{equation}
for a positive constant $c_{n}$ depending only on $n$.

We recall that we have fixed an embedding of $\Gamma$ in some $\GL_{n}(\ZZ)$.
Let $A$ be the maximum of the heights of elements of $ S_{\cF}$.
Then an element $\gamma\in \Gamma$ with $l(\gamma)\leq N$
satisfies $H(\gamma)\leq (nA)^{N}$.

We recall that for any positive real $T$ we defined 
$$
\Theta(\Sigma_{Y},T):=\{g\in G(\QQ)\cap \Sigma(Y),\ H(g)\leq T\}
$$
 and
 $N(\Sigma(Y),T)=\vert \Theta(\Sigma_{Y},T)\vert$.

\begin{lem}\label{l5.7}
There exists a constant $c_{1}>0$ such that for all $T$ big enough
$$
\{\gamma\in \Gamma\cap \Sigma(Y), H(\gamma)\leq T \}\ge T^{c_{1}}.
$$
As a consequence
\begin{equation}
N(\Sigma(Y),T)\ge T^{c_{1}}.
\end{equation}
\end{lem}
\begin{proof}
Let $N$ be an integer. The previous discussion shows that the set
$\{\gamma\in \Gamma\cap \Sigma(Y), l(\gamma)\leq N \}$
is contained in 
$\Theta(\Sigma_{Y},(An)^{N})$. The result is therefore
an application of proposition \ref{p.5.3}
\end{proof}

We can now give the proof of theorem \ref{t.5.5}. Let $c_1$ be the consant from lemma \ref{l5.7}.
Using the theorem of Pila and Wilkie \ref{PW} with $\epsilon = \frac{1}{2n}$, we know that 
for all $T$ big enough, 
$\Theta(\Sigma_{Y},T^{\frac{1}{2n}})$ is contained
in at most $T^{\frac{c_{1}}{4n}}$ semi-algebraic blocks.
Using lemma \ref{l5.7}, we see that there exists a semi-algebraic block
$W$ of $\Sigma(Y)$ containing at least $T^{\frac{c_{1}}{4n}}$
elements $\gamma\in \Sigma(Y)\cap \Gamma$ such that $H(\gamma)\leq  T^{\frac{1}{2n}}$.
Using lemma \ref{l.5.6}
we see that there exists $\sigma\in \Sigma(Y)$ such that
$\sigma \Theta_{Y}$ 
 contains at least $T^{\frac{c_{1}}{4n}}$
elements $\gamma\in \Sigma(Y)\cap \Gamma$ such that $H(\gamma)\leq  T^{\frac{1}{2n}}$.

Let $\gamma_{1}$ and $\gamma_{2}$ be two elements of $\sigma \Theta_{Y}\cap \Gamma$
such that $H(\gamma)\leq  T^{\frac{1}{2n}}$. Then using the equations (\ref{eq.5}) and (\ref{eq.6})
we see that 
$\gamma:=\gamma_{2}^{-1}\gamma_{1}$ is an element of $\Gamma\cap \Theta_{Y}$
such that $H(\gamma)\leq nc_{n}T^{1/2}$.  Therefore for all $T$ big enough
$\Theta_{Y}$ contains at least $T^{\frac{c_{1}}{4n}}$ elements  $\gamma\in \Gamma$ 
such that $H(\gamma)\leq T$. 
 As $Y$ is algebraic, the stabilizer $\Theta_{Y}$ is an algebraic
group and the group $H_{Y}$  generated by the Zariski closure of $\Gamma\cap \Theta_{Y}$
is a $\QQ$-algebraic subgroup of $G$  of positive dimension contained in $\Theta_{Y}$.
As $G_{\QQ}$ is $\QQ$-anisotropic, $G(\QQ)$ doesn't contain unipotent elements
and $H_{Y}$ is therefore a reductive  $\QQ$-group.

\section{End of the proof of the theorem \ref{main_thm} using Hecke correspondences.}

In this section we prove theorem \ref{main_thm}.

Let us recall the situation. We have a Hodge generic subvariety $V$ of $S$
and $Y$ a maximal algebraic subvariety of $\wt{V}= \pi^{-1}V$.
We assume that $Y$ is Hodge generic.

The group $G$ is semisimple of adjoint type:
$$
G = G_1 \times \cdots \times G_r
$$
and, accordingly,
$$
X^+ = X_1^+ \times \cdots \times X_r^+.
$$
Let $p_{i}:G\longrightarrow G_{i}$ be the projection on $G_{i}$.
Suppose that $V\not=S$ (otherwise there is nothing to prove).
By lemma \ref{reduction2} we assume that $V$ is not of the form
$$
S_1\times \cdots \times S_k \times V'
$$
where $V'$ is a subvariety of $S_{k+1}\times \cdots \cdots \times S_r$.
Recall that $H_Y$ is a reductive group.
After, if necessary, reodering the factors, we let $G_1, \dots,G_t$ 
be the factors to which $H_Y$ projects non-trivially.

\begin{lem}\label{1.6.1}
Let $C(V)$ be the Nori constant of $V$. There exists $p> C(V)$ and $g\in H_{Y}(\QQ)$ such that
for any $l\not= p$, $g_l\in K_l$ and  for $i = 1,\dots , t$, the element  $p_{i}(g_p)$ is not contained in compact
subgroups.
\end{lem}
\begin{proof}  Let $T$ be a maximal $\QQ$-torus in $H_{Y}$.
 Pick a prime $p>C(V)$ such that $T_{\QQ_p}$ is a split torus.
 For $i=1,\dots, t$, the torus $T_{i}:=p_i(T)$ is a maximal in $G_i$. Moreover, each $T_{i\QQ_p}$ is split.
Let $X_*(T)$ be the cocharacter group of $T$. As $T_{\QQ_p}$ is split, we have $T(\QQ_p) = X_*(T)\otimes \QQ_p$
and the valuation map $v_p\colon \QQ^* \lto \ZZ$ induces an isomorphism
$$
\psi:\ T(\QQ_p)/K^m_{T,p} \cong X_*(T)
$$
where $K^m_{T,p}$ is the maximal compact open subgroup of $T(\QQ_p)$. 
Let $X_*(p_i)$ be the map $X_*(T)\lto X_*(T_i)$ induced by $p_i$.
The kernels of $X_*(p_i)$ give $t$ proper subroups of $X_*(T)$. We choose an element $a_p \in T(\QQ_p)$ such that the  $\psi(a_{p})K_{T,p}^m\in X_*(T)$
 avoids these subgroups.
Let $\alpha$ be the element of $T(\AAA_{f})$ such that $\alpha_{p}=a_p$ and for all $l\neq p$ $\alpha_{l}=1$.
For all $l\neq p$,  $\alpha_l\in K_l$ and   for $i = 1,\dots , t$
 the images of $\alpha_p$ in $T_i(\QQ_p)$ are not contained in compact subgroups.
 
 Let $K_{T}^{m}$ be the maximal compact open subgroup of
 $T(\AAA_{f})$.
 As the class group $T(\QQ)\backslash T(\AAA_{f})/K_{T}^{m}$ is finite
 there exist an integer $s$, such that 
 $\alpha^{s}= g k$ where $g\in T(\QQ)$ and $k\in K^m_T$.
For $l\not= p$
$g_l\in K_{T,l}^m$  and for   $i = 1,\dots , t$, the images of $g_p$ in $G_i(\QQ_p)$ are not contained in compact open subgroups.
 For $i>t$, $p_{i}(T)$ is trivial. This finishes the proof of lemma \ref{1.6.1}. 
 
\end{proof}

By properties of the Nori constant $C(V)$, the element $g$ of \ref{1.6.1} satisfies the assumptions of
propositions \ref{irred} and \ref{characterisation}.

As
$$
Y = g Y,
$$

$$
\pi(Y)\subset V \cap T_{g}V.
$$

By the choices we made and proposition \ref{irred} the varieties $T_gV$  and $T_{g^{-1}}V$ are irreducible.
Suopose that $V = T_gV$. Then $V=T_{g^{-1}}V$. 
Indeed, $V = T_{g}V$ if and only if $\gamma_1 g \gamma_2 \cdot \wt{V} = \wt{V}$ for all $\gamma_1,\gamma_2\in \Gamma$.
This is equivalent to  $\gamma_1 g^{-1} \gamma_2 \cdot \wt{V} = \wt{V}$ for all $\gamma_1,\gamma_2\in \Gamma$.
By proposition \ref{characterisation}, we conclude that
$$
V = S_1\times \cdots S_t \times V'
$$
where $V'$ is a subvariety of $S_{t+1}\times \cdots \times S_r$.
This finishes the proof in the case where $V = T_gV$.

Suppose now that the intersection $V$ and $T_{g}V$ is not proper. As $T_gV$ is irreducible, the intersection
$V\cap T_gV$ is not proper.
Let $Z = V\cap T_{g}V$. Then $Y\subset \pi^{-1}Z$.
As $Y$ and $\pi^{-1}Z$ are analytic varieties, there exists an analytic component 
containing $Y$. We let $V_1$ be the image of this component by $\pi$.
As $V_1$ contains $\pi(Y)$, it is Hodge generic.
Notice that the fact that the images of all projections of $Y$ are positive dimensional 
implies that the projections of $V_1$ are positive dimensional (in particular the monodromy is maximal).
We than choose a prime $p>C(V_1)$ and reiterate the process.
Notice that we may, by lemma \ref{reduction2} again assume that $V_1$ is not of the form $S_1\times \cdots \times S_k \times V'$. 
This way, we construct a decreasing sequence $V_i$ of subvarieties.  
A dimension argument shows that we end up with  $V_k= S_1\times \cdots \times S_t \times V'$ as above.
This gives a contradiction, hence $V_k=S$. This finishes the proof.

\end{document}